\documentclass[10pt]{amsart}
\usepackage{amsmath,amsthm,amssymb,latexsym}
\usepackage{epsfig}
\usepackage{graphics}

\newtheorem{theorem}{Theorem}
\newtheorem{remark}{Remark}
\newtheorem{lemma}{Lemma}

\newtheorem{definition}{Definition}

\begin{document}
\title[Cable expansion]
{On the cable expansion formula}
\author{Qihou Liu}
\email{qihouliu@aya.yale.edu}

\date{Sep 15,2008}
\maketitle

\begin{abstract}
 In this paper, a generalized version of Morton's formula \cite{MO} is proved. Using this formula, one can write down the colored Jones polynomials of cabling of an knot in terms of the colored Jones polynomials of the original knot.
\end{abstract}

\section{Introduction}\label{intro}
In \cite{MO}, Morton proved a formula relating the colored Jones polynomials of a cabling of a knot to the colored Jones polynomials of the original knot. Using this formula, Kashaev proved that the volume conjecture is true for the torus knots \cite{KT}.

The original proof of Morton is written in the language of quantum groups. In this paper, an alternative way using skein theory is adopted. Although the method of quantum groups is essentially dual to the method of skein theory, the latter is more accessible to topologists. Unlike quantum groups, the skein theory starts from the pure combinatorial Kauffman's relations to build the whole theory. The relation to representation theory is hidden. In this sense, it is a straightforward approach.

In a recent preprint \cite{VV}, a third version of the expansion formula is obtained.

\section{Basic definitions and notations}\label{def}

To keep the prerequisite minimal, some basic definitions and notations are listed in this section. For
details and proofs of certain results, see Lickorish's book \cite{LBook} and \cite{BHMV}. For a
survey of the skein theory and some generalization,
see \cite{LSurvey}.

Given a compact oriented 3-manifold $M$, and a ring of coefficients
$R=Z[A,A^{-1}]$, define a $R-module$ by
\[SS(M)=the\ free\ module\ generated\ by\ framed\ links\ in\ M\]
where two framed links isotopic to each other are regarded as the same.

\begin{definition} The skein module associated to $M$ is
\[S(M)= SS(M)/the\ Kauffman\ skein\ relations\]
\end{definition}

These skein relations are the following:

\[\scalebox{.15}{\psfig{file=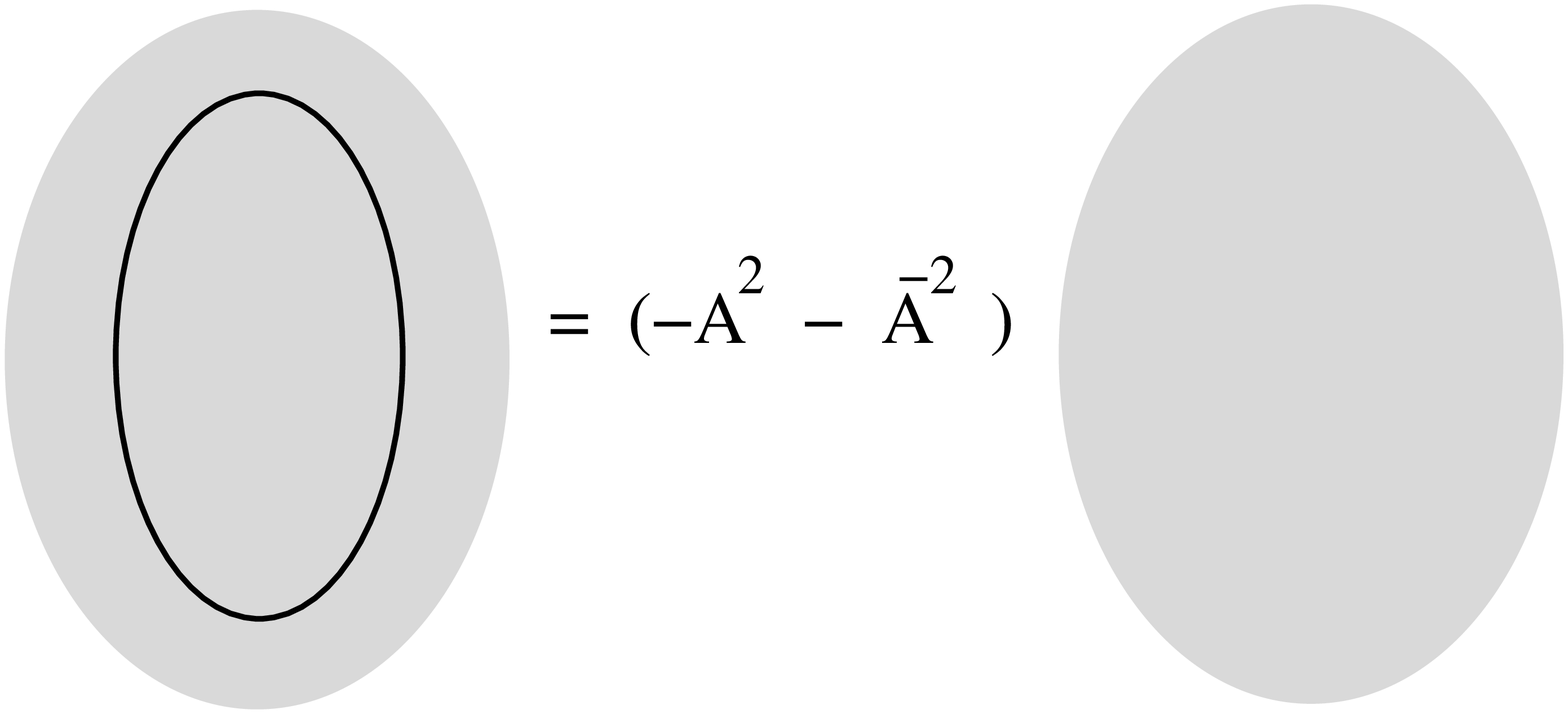}}\]

\[\scalebox{.15}{\psfig{file=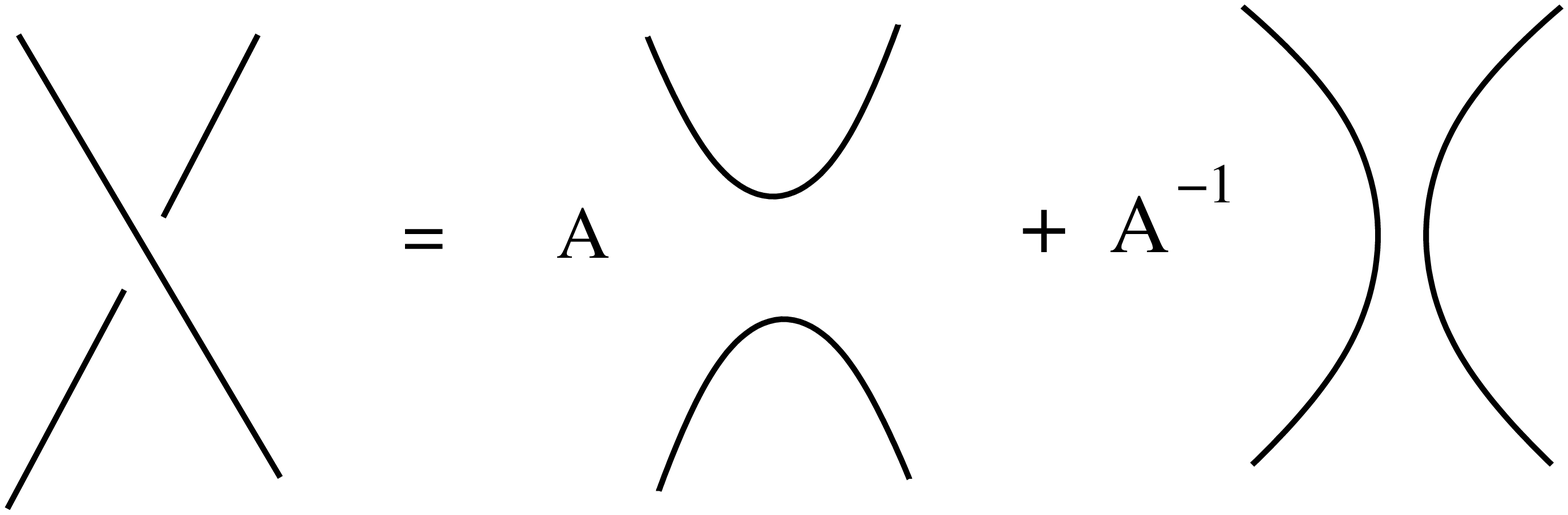}}\]

\subsection{Examples}\label{examples}

\subsubsection{The skein module of $S^3$}\label{three sphere}
The skein module of $S^3$ is equal to $Z[A,A^{-1}]$. The reason is the following.

Given a planar diagram of a famed link $L$ in $S^3$, by applying the skein relation repeatedly to resolve the crossings and remove the unknots, $L$ can be reduced to the empty link multiplied by a Laurent polynomial of $A$. In this way, a polynomial is associated to each framed link in $S^3$.

Multiplied by a normalization factor, the dependence of the polynomial on the framing can be removed. Up to change of variables, this is the Jones polynomial of $L$.

\subsubsection{The skein module of $D\times S^1$}\label{solid torus}
The case of the solid torus is slightly different. By simply using the skein
relations, not every link can be reduced to the empty one. Certain
essential links remain.

$S(D\times S^1)$ is found to be a free  $R-module$
generated by those elements represented by parallel strings along
the core of the solid torus, see the pictures below.
\[\scalebox{.35}{\psfig{file=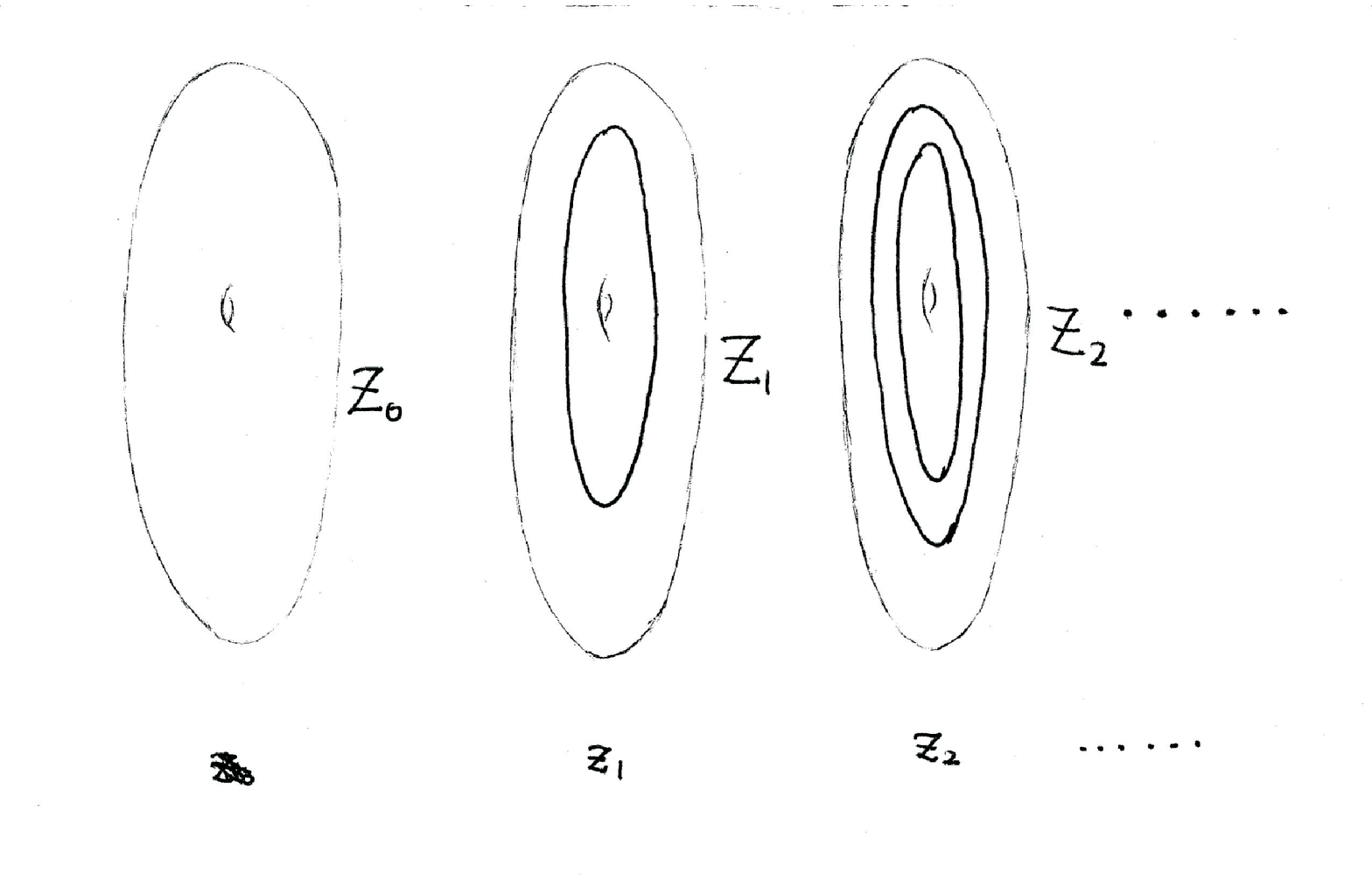}}\]
\subsubsection{The algebra structure of $S(D\times S^1)$ and an alternative basis}\label{algebra and new basis}
The skein module of the standard solid torus $D\times S^1$ has an
algebra structure. The multiplication is defined as follows.

Fix a curve on $\partial(D\times S^1)$, for instance $\{1\}\times S^1$. Here $D$ is identified with the unit disk in the complex plane. Name this curve $C_+$.
Choose another curve on $\partial(D\times S^1)$ parallel to $C_+$, for instance
$\{-1\}\times S^1$, name it  $C_-$. Given two copies of $D\times
S^1$, glue the first one to the second one by identifying a neighborhood
of $C_+$ in the boundary from the first copy to a neighborhood of
$C_-$ in the boundary from the second copy. The result is still a solid
torus.

Two elements from $S(D\times S^1)$ which can be represented by skeins are multiplied by juxtaposition of the two skeins representing them. See the picture below.
\[\scalebox{.45}{\psfig{file=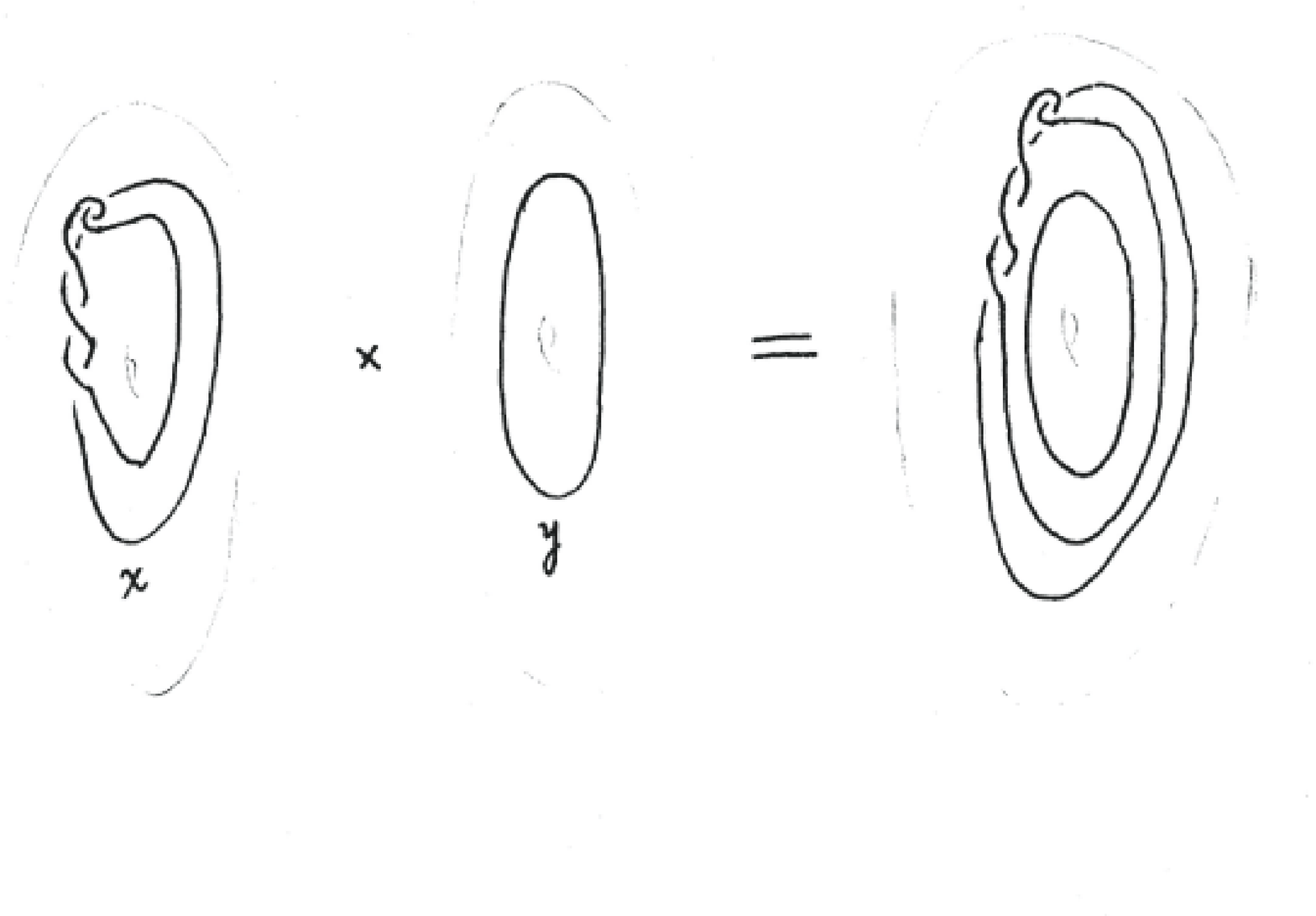}} \]

The multiplication is then extended linearly. With this
multiplication, $S(D\times S^1)$ is a commutative $R-algebra$.
It is isomorphic to $R[z_1]$, where $z_1$ is one of the basis
elements described above. Note that $z_n=z_1^n$.

 Define elements $\{e_n|n\ integers\}$ of $S(D\times S^1)$ by $e_0=1,e_1=z_1$ and $z_1e_n=e_{n-1}+e_{n+1}$.
These relations determine $e_n$ for $n\ge 0$ as well as those $e_n$ where $n<0$. It can be checked by induction that $e_n=-e_{-n-2}$.

The set $\{e_n|n\ge 0 \}$ is a basis of
$S(D\times S^1)$, since each $e_n$ is a monic polynomial of $z_1$ with degree n.

It is worth mentioning the description of the multiplication using the new basis.

\begin{definition} A triple of nonnegative integers $(l,m,n)$ is said to be admissible if $l+m+n$ is even and $(l,m,n)$ satisfies the triangle inequality, that is \[ |m-n|\le l \le m+n\]
\end{definition}

The multiplication of the basis elements is given by
\[ e_m e_n=\sum_{(l,m,n)\ admissible} e_l\]
for $m,n$ nonnegative.
This can also be checked by induction on n.

\subsection{The multilinear bracket}\label{multilinear bracket}
Given an framed link $L$ in $M$ with components $K_1,K_2,...K_m$, denote the framing of $K_j$ by $\sigma_j$. A multilinear
bracket $\langle *,*,...,*\rangle_L$ in $S(M)$ can be constructed as following:

\

The framing $\sigma_j$ of $K_j$ determines a diffeomorphism
\[\tau_j:D\times S^1 \to a\ tubular\ neighbourhood\ of\ K_j.\]
$\tau_j$ gives a map $D\times S^1 \to M$. It induces a homomorphism
\[(\tau_j)_*:S(D\times S^1) \to S(M).\]
These $(\tau_j)_*$ combine
together to give a multi-linear map \[\tau_*: S(D\times
S^1)\times S(D\times S^1)\times ...\times S(D\times S^1)\to S(M)\]
where the product contains $m$ factors.

Define $\langle
x_1,x_2,...,x_m\rangle_L$ to be $\tau_*(x_1,x_2,...,x_m)$.

\

Components of $L$ can be colored by integers. Assume $K_j$ is colored by $n_j$, then the colored framed link $L$ gives an element in $S(M)$, that is $\langle e_{n_1},e_{n_2},...,e_{n_m}\rangle_L$.

\begin{remark}
Change the framing of certain components of $L$, while keeping the colors unchanged, for example $\sigma_1$ is changed to $\sigma_1+b$, where $b$ is an integer. Then
\[\langle e_{n_1},e_{n_2},...,e_{n_m}\rangle_{L'}=(-1)^{bn_1}A^{b(n_1^2+2n_1)}\langle e_{n_1},e_{n_2},...,e_{n_m}\rangle_L\]
where $L'$ is the same link as $L$ with changed framings.
\end{remark}

\subsubsection{Examples}\label{examples}
 Let $K$ be a knot in $S^3$ with framing $\sigma$, the n-th colored Jones polynomial of K, $J_n^{\sigma}(A,K)$, is defined to be $\langle e_n
\rangle_K$,which is an element of $S(S^3)$. Since $S(S^3)$ is equal to $Z[A,A^{-1}]$, this element is a Laurent polynomial in A.

 In particular,
\[\langle e_n\rangle_U=(-1)^n[n+1]\]
where $U$ is the unknot with framing zero, and $[m]=\frac{A^{2m}-A^{-2m}}{A^2-A^{-2}}$.

Let L be the following link in $D\times S^1$:
\[\scalebox{.25}{\psfig{file=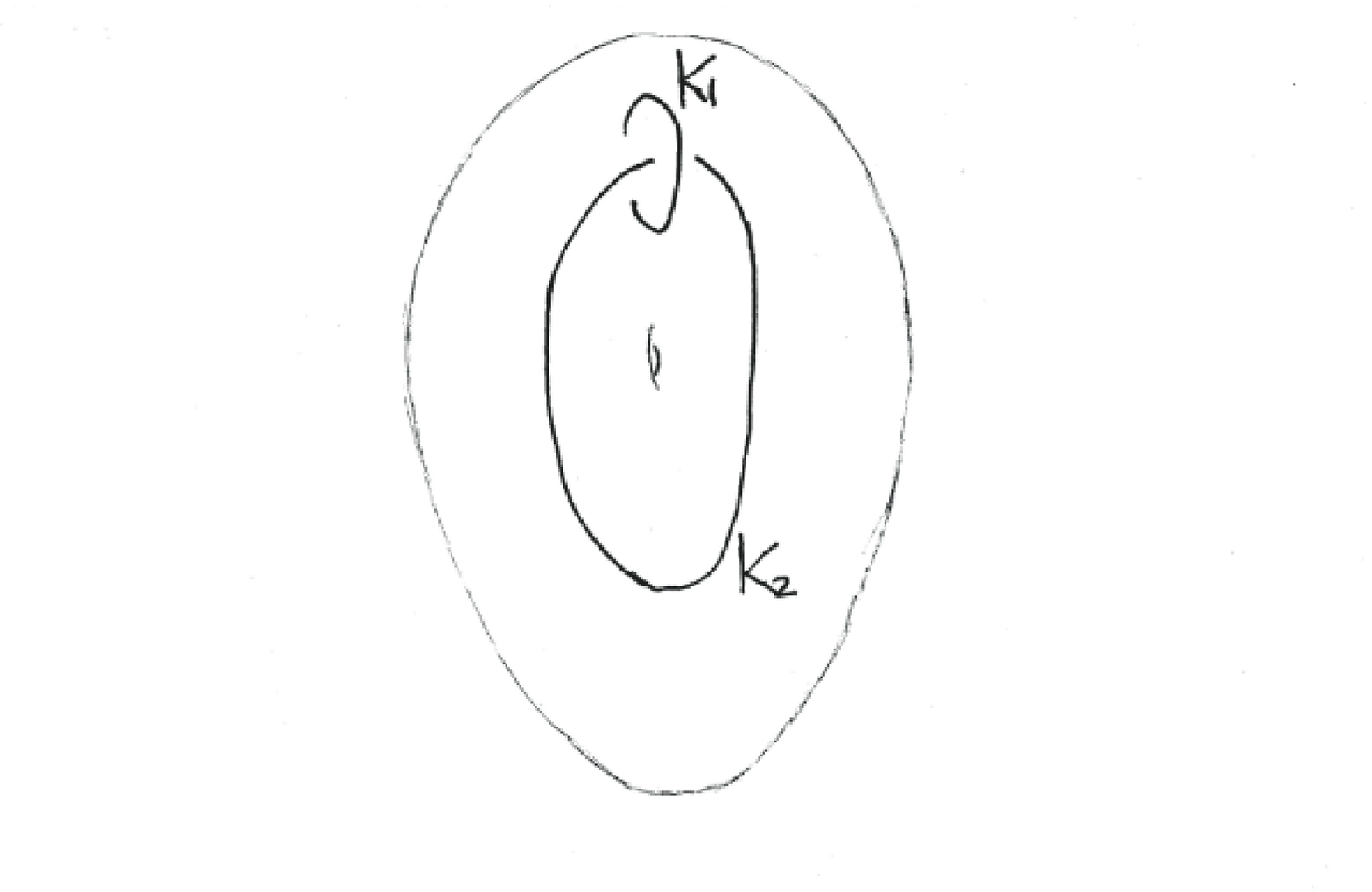}}\]

Here both components have framing 0, --- consult the framing convention below. Then we have \[\langle e_m,e_n\rangle_L=(-1)^m \frac{[(m+1)(n+1)]}{[n+1]} e_n\]

If $K$ has framing 0, $J_n(A,K)$ is used for the colored Jones polynomial instead of $J_n^0(A,K)$.

\subsubsection{The framing convention}
Since the linking number of two oriented knots in $S^3$ is defined, the framing of a knot in $S^3$ can be specified by an integer,  the linking number.

For an arbitrary 3-manifold $M$, the framing of a knot $K$ in $M$ can't be specified by an integer.
It is only possible under certain circumstances.

For example, $M$
is the solid torus $D\times S^1$ with a chosen curve $C_+$,
$\{1\}\times S^1$, on its boundary. Embed $M$ into $S^3$. A
framed knot $K$ in $M$  is then mapped to a framed knot in $S^3$. Its
framing corresponds to an integer. In general the corresponding
integer will change if the embedding is changed. If $D\times S^1$ is embedded into  $S^3$ in such a way that $C_+$ has linking
number zero with the core of the solid torus after the embedding, the corresponding integer will be the same for all such embeddings.

In the following, a
knot $K$ in $D\times S^1$ is said to have framing n for some integer n in the
sense that the knot is viewed as a knot in $S^3$ through such kind
of embedding of  $D\times S^1$ into $S^3$.

\section{Expansion of cables in the skein module}\label{thm}
\label{thm}
 In this section, the framing convention is assumed.

 Let $M$ be the standard solid torus $D \times S^1$, the link $L$, which is to be considered, consists of two components:
 $K_1$ is given by \[( \rho \exp(iq\theta), \exp(ip\theta))|_{0\le \theta \le 2\pi}\]
 where $p,q$ are coprime integers, $0<\rho<1$;
 $K_2$ is the core of the solid torus
 \[(0, \exp(i\theta))|_{0\le \theta \le 2\pi}\]
 Here $D$ is identified with the unit disk in the complex plane and $S^1$ with the unit circle.

  Fix the framing of $K_2$ to be 0.  The framing $\sigma$ of $K_1$ corresponds to some integer. By abuse of notation, it is also denoted by $\sigma$. Denote the color of $K_1$ by $N$. Denote the color of $K_2$ by $s$. Assume they are nonnegative.  $\langle e_N,e_s\rangle_L$ is some element of $S(D \times S^1)$. Denote it by $T_\sigma^N(p,q;s)$.

 As an element of $S(D \times S^1)$, $T_\sigma^N(p,q;s)$ can be expressed as a linear combination of the basis elements $\{e_l\}|_{l\ge 0}$. The coefficients before $e_l$ is denoted by
 $g_\sigma^l(p,q;s)$, i.e.
 \[T_\sigma^N(p,q;s)=\sum_{l\ge 0} g_\sigma^l(p,q;s)e_l\]

 Note the range of $l$ is not restricted except that $l\ge 0$. These coefficients
 are all 0 but for finitely many $l$'s.
 
 First some simple cases:

\begin{enumerate}
        \item If (p,q)=(0,1),then it follows from the second example in 2.1.2 that with $\sigma=0$
                  \[g_\sigma^l(0,1;s)=\delta_{ls}(-1)^N\frac{[(N+1)(s+1)]}{[s+1]}\].
        \item  If (p,q)=(1,0), we have two parallel strings, one with color s, the other with color N. It's clear from the multiplication formula of $e_n$'s that for framing $\sigma=0$, $g_\sigma^l(1,0;s)=1$
                   if (l,s,N) is an admissible triple, and $g_\sigma^l(1,0;s)=0$ otherwise.
                   By introducing the $\delta$-function, which is $\delta(0)=1$ and $\delta(m)=0$ for nonzero m, we can encode these phrases into a formula
                   \[g_\sigma^l(1,0;s)=\sum_{ \substack {k=-N\\ k+N\ even}}^N \{\delta(l-s-k)-\delta(l+s+2+k)\}\]
        \item If we change the framing $\sigma$ of the cable, by the remark in 2.1.2, we just add an overall phase factor to each of the $g^,$s, i.e.
                   \[g_{\sigma+b}^l(p,q;s)=(-1)^{bN}A^{b(N^2+2N)}g_\sigma^l(p,q;s)\]
        \item Since (-p,-q) represents the same cable as (p,q) does, we shall always assume $p>0$.
\end{enumerate}

The approach to calculate $g^,$s for general (p,q) are quite straightforward. We start from (1,0), then calculate the case(1,q),
then (q,1), then (q,$1+bq$), then ($1+bq$,q)...\ Formally, we express $\frac{q}{p}$ as a continued fraction
\[\frac{q}{p}=a_0+\frac{1}{a_1}\ _{+}\ \frac{1}{a_2}\ _{+}\ ...\ _{+}\ \frac{1}{a_\gamma}\]
and then use induction on $\gamma$.

This is made possible by the following two lemmas.

\begin{lemma}
\[g_\sigma^l(p,q+p;s)=(-1)^{l+s}A^{-s^2-2s+l^2+2l}g_{\sigma^\prime}^l(p,q;s)\]
\end{lemma}

\begin{proof} Since to get the (p,q+p)-cable from a (p,q)-cable, one simply apply Dehn twist once around the meridian. And it
follows from the remark in 2.1.2. The phase factor appear because twisting has changed the framing of the core string of the solid torus as well.
\end{proof}

 Now how to make transition from (q,p) to (p,q)? We use a multi-linear bracket associated to the following
 link in $S^3$: $L=K_1\cup K_2\cup K_3$

\

\ 

\ 

\ 

\[\scalebox{.25}{\psfig{file=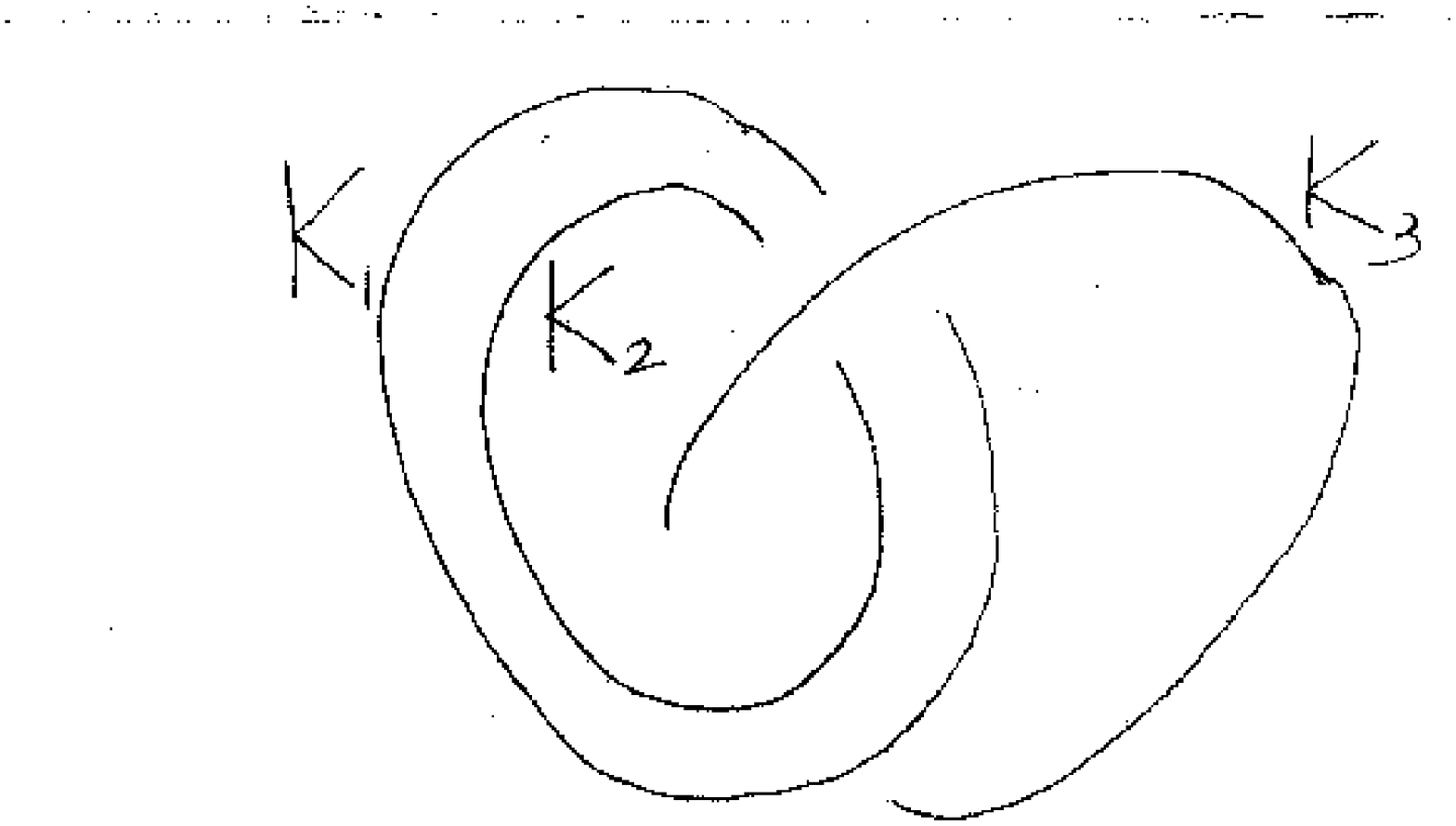}}\] each component has 0
framing.  Let $H$ be the Hopf link. We have
\[\langle x_1,x_2, x_3\rangle_L=\langle x_1\cdot x_2,x_3\rangle_H\]
where $x_1\cdot x_2$ is the multiplication in $S(D\times S^1)$.

Let $r$ be a positive integer, and
\[\omega_r=\sum_{t=0}^{r-1} (-1)^t[t+1]e_t\]
then for $r>m>0$, we have
\[\langle e_m,\omega_r\rangle_H\equiv 0\]
and
\[\langle e_0,\omega_r\rangle_H\equiv \langle\omega_r\rangle_U\]
where $U$ is the unknot with framing 0, equivalence means equal at
$A=exp\left(\frac {\pi i}{2(r+1)}\right)$. See Lemma 13.9 from
Lickorish's book~\cite{LBook}.

Thus for $r>2l\ge0,r>2k\ge0$, we have
\begin{align}\langle e_l,e_k,\omega_r\rangle_L&=\langle e_l\cdot e_k,\omega_r\rangle_H\notag\\
                       &=\sum_{\substack{m\\(k,l,m)\ admissible}}\langle e_m,\omega_r\rangle_H\notag\\
                       &\equiv \sum_{\substack{m\\(k,l,m)\ admissible}}\delta(m)\langle\omega_r\rangle_U\notag\end{align}
 Thus we have
 \[\langle e_l,e_k,\omega_r\rangle_L\equiv\delta(l-k)\langle\omega_r\rangle_U\tag {*}\]

Then \[\langle e_l,\ T_\sigma(p,q;s),\omega_r\rangle_L\equiv g_\sigma^l(p,q;s)\langle\omega_r\rangle_U\tag {**}\]

Let $L_1$ be the second example mentioned in 2.1.2, we have
\[\langle e_l,e_t\rangle_{L_1}=(-1)^l \frac{[(l+1)(t+1)]}{[t+1]} e_t\]
Map $D \times S^1$ to a tubular neighborhood of $K_3$, we have
\begin{align} \langle e_l,\ T_\sigma(p,q;s),\omega_r\rangle_L&=\langle T_\sigma(p,q;s),\langle e_l,\omega_r\rangle_{L_1}\rangle_H\notag\\
&=\langle T_\sigma(p,q;s),\ \sum_{t=0}^{r-1}
(-1)^t[t+1]\langle e_l,e_t\rangle_{L_1}\rangle_H\notag\\
&=\langle T_\sigma(p,q;s),\ \sum_{t=0}^{r-1} (-1)^{l+t}[(t+1)(l+1)]
e_t\rangle_H\tag {***}
\end{align}

\

\ 

\ 

\ 

\[ \scalebox{.25}{\psfig{file=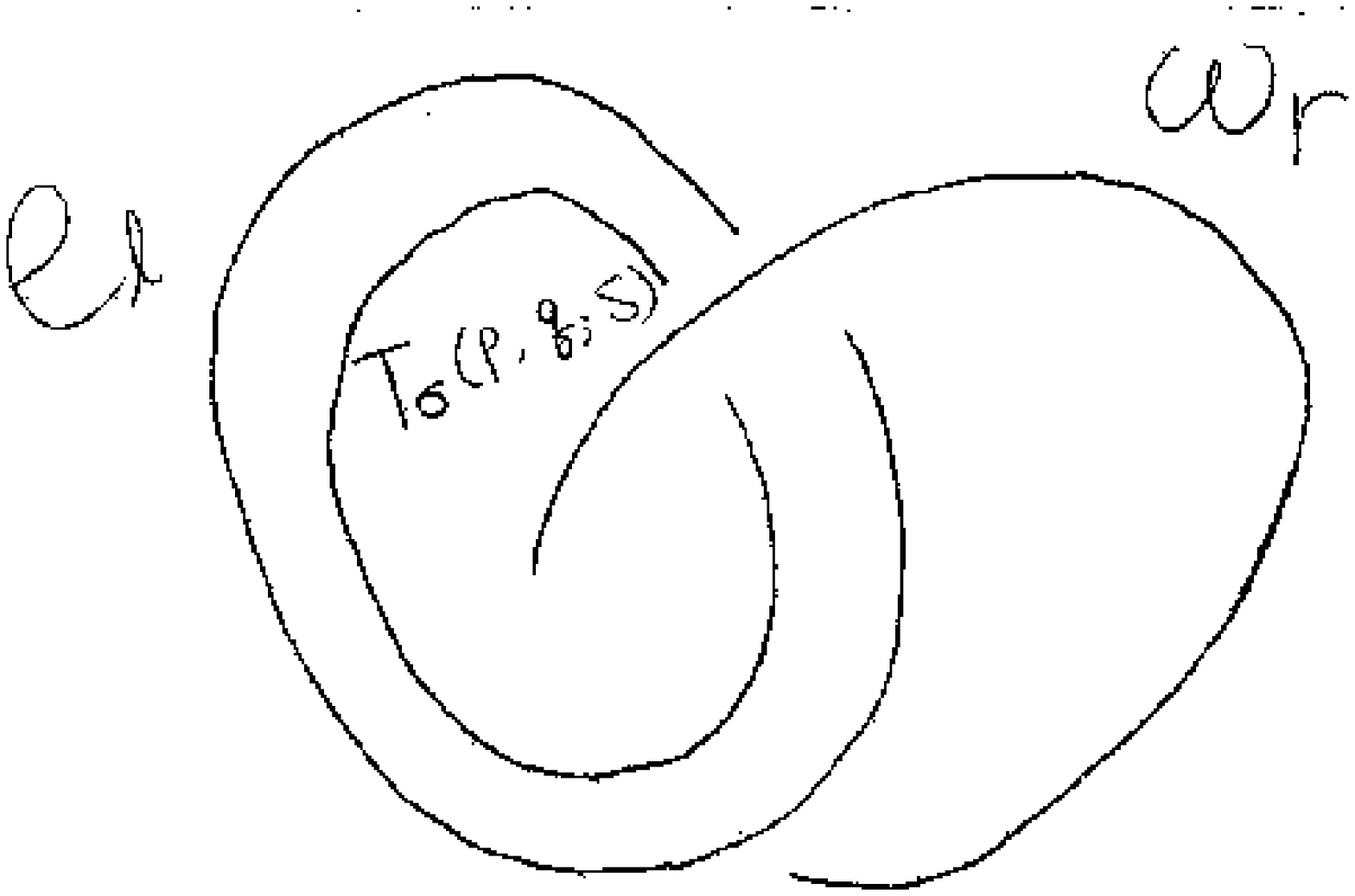}} =
\scalebox{.25}{\psfig{file=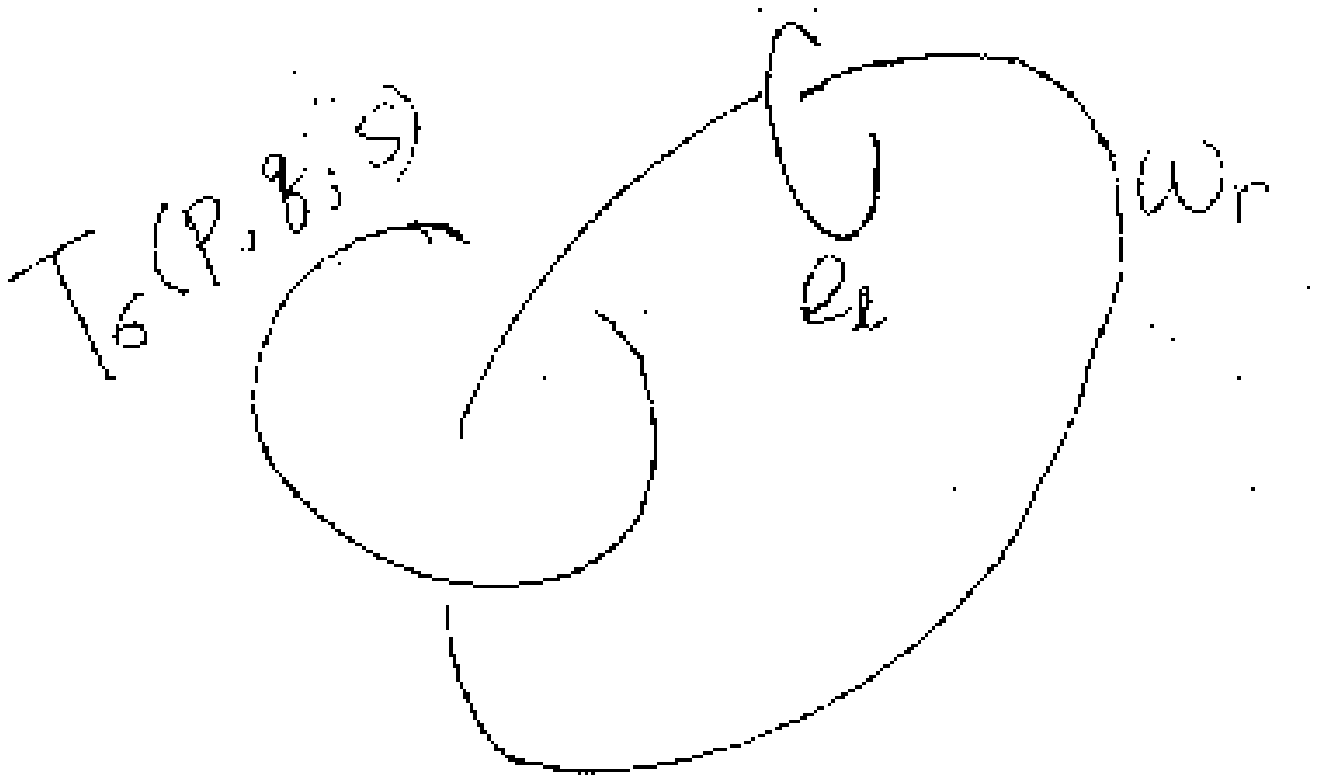}}
\]

Now we can go on. We have a link $T_\sigma(p,q;s)$ in the tubular neighborhood $N(K_2)$ of the second component $K_2$ of L,
which consists of a string at the core of the solid torus, surrounded by a (p,q)-cable. Let's push this (p,q)-cable off to the
boundary, and push it a little further so that it slides into the world on the other side of the boundary torus. Observe that
$S^3-N(K_2)$ is still a solid torus, since $K_2$ is the unknot. The meridian of the $N(K_2)$ becomes a longitude when viewed
from the other side, and the longitude we chose becomes a meridian of $S^3-N(K_2)$. And we push the cable further into $N(K_3)$,
then the (p,q)-cable becomes (q,p)-cable in $N(K_3)$.

So
\[
\langle T_\sigma(p,q;s),\ \sum_{t=0}^{r-1} (-1)^{l+t}[(t+1)(l+1)]e_t\rangle_H=\langle e_s,\ \sum_{t=0}^{r-1}  (-1)^{l+t}[(t+1)(l+1)]
T_{\sigma}(q,p;t)\rangle_H\tag {****}
\]
Then we expand $T_{\sigma}(q,p;t)$ on the right hand side, note that \[\langle e_s,e_k\rangle_H=(-1)^{s+k}[(s+1)(k+1)]\]
in this way get a relation between $g_\sigma^l(p,q;s)$ and
$g_{\sigma}^k(q,p;t)$.

We summarize (**),(***),(****) into the following lemma.

\begin{lemma}
\begin{align} (-1)^{l+s}&g_\sigma^l(p,q;s)\langle\omega_r\rangle\notag\\
&\equiv \sum_{t=0}^{r-1}\sum_{k\ge 0} (-1)^{k+t}g_{\sigma}^k(q,p;t)[(t+1)(l+1)][(s+1)(k+1)]\notag\end{align}
here equivalence means equal at $A=exp\left(\frac {\pi i}{2(r+1)}\right)$. Note that in this formula $g(p,q)$ and $g(q,p)$ have the same framing.
\end{lemma}

Using the two lemmas and the formula for (p,q)=(1,0), we can now write down the
general formula for $g^,$s.

\begin{theorem} Assume $p>0$,let $\epsilon, \beta, \alpha$ be such that $\epsilon=[\frac{q}{p}]$, $q=\epsilon p+\beta$ with
$0\le \beta<p$, $\alpha=p\cdot\beta$.Then for framing $\sigma=pq$, we have
\begin{align} g_\sigma^l(p,q;s)&=(-1)^{qN} A^{\epsilon
(-s^2-2s+l^2+2l)}\notag\\
&\times\sum_{ \substack {k=-N\\ k+N\ even}}^N A^{\alpha {k}^2+2\beta k
(s+1)}\{\delta(l-s-pk)-\delta(l+s+2+pk)\}\notag\end{align}
\end{theorem}

It seems to be complicated. But note there're at most two nonzero terms in this summation. Since for each value of $k$, the
summand is nonzero only if $l-s-pk=0$ or $l+s+2+pk=0$.

\begin{theorem} Substitute the expression for $g$'s to the expansion of $T_\sigma(p,q;s)$, we have
\[T_\sigma(p,q;s)=(-1)^{qN}\sum_{ \substack {k=-N\\ k+N\ even}}^N A^{pqk^2+2qk(s+1)} e_{pk+s}\]
\end{theorem}

\begin{proof}[Proof of theorem 2] Substitute the expression for $g$'s, then change the order of summation. Also note that
$e_l=-e_{-l-2}$ and $e_{-1}=0$.
\end{proof}

\section{Proof of Theorem 1}\label{proof}
\label{proof}

\begin{proof}
\

Assume it's true for (p,q), let's do the calculation in the case (q,p).

By explicit calculation, we got
\[\langle\omega_r\rangle_U(A^2-A^{-2})^2\equiv -2(r+1)\]

So by lemma 2,
\begin{align} &(-1)^{l+s}g_\sigma^l(q,p;s)(-2(r+1))\notag\\
&= 4\sum_{t=0}^{r-1}\sum_{k\ge 0} (-1)^{k+t}g_{\sigma}^k(p,q;t)\sinh((t+1)(l+1)C)\sinh((s+1)(k+1)C)\notag\end{align}
 where we evaluate these polynomials at $A=\exp(\frac{\pi i}{2(r+1)})$, and we denote the constant $\frac{\pi i}{r+1}$ by $C$.

And then using the formula for $g(p,q;t)$, we have
\begin{align} &(-1)^{l+s}g_\sigma^l(q,p;s)(-2(r+1))\notag\\
&= (-1)^{qN}4\sum_{t=0}^{r-1}\sum_{k\ge 0}
(-1)^{k+t} \exp(\epsilon(-t^2-2t+k^2+2k)C/2)\notag\\
&\qquad\qquad\qquad\times\sinh((t+1)(l+1)C)\sinh((s+1)(k+1)C)\notag\\
&\times\sum_{ \substack {k'=-N\\ k'+N\ even}}^N \exp((\alpha {k'}^2+2\beta k'
(t+1))C/2)\{\delta(k-t-pk')-\delta(k+t+2+pk')\}\notag\end{align}

Note that $g_{\sigma}^k(p,q;t)$ is 0, unless $k+t+pN$ is even, so $(-1)^{k+t}=(-1)^{pN}$.

So
\[(-1)^{l+s}g_\sigma^l(q,p;s)(-2(r+1))=(-1)^{qN}(-1)^{pN}\times S \tag{A}\]

where
\begin{align} &S=4\sum_{t=0}^{r-1}\sum_{k\ge 0} \exp(\epsilon(-t^2-2t+k^2+2k)C/2)\notag\\
&\qquad\cdot\sinh((t+1)(l+1)C)\sinh((s+1)(k+1)C)\notag\\
&\qquad\quad\cdot\sum_{ \substack {k'=-N\\ k'+N\ even}}^N \exp((\alpha {k'}^2+2\beta k'
(t+1))C/2)\{\delta(k-t-pk')-\delta(k+t+2+pk')\}\notag\\
&=4\sum_{ \substack {k'=-N\\ k'+N\ even}}^N \sum_{t=0}^{r-1}\exp(\alpha {k^\prime}^2C/2+\beta
k^\prime(t+1)C)\sinh((t+1)(l+1)C)\notag\\
&\qquad\cdot \sum_{k\ge 0}
\exp(\epsilon(-t^2-2t+k^2+2k)C/2)\sinh((s+1)(k+1)C)\notag\\
&\qquad\quad\cdot\{\delta(k-t-pk')-\delta(k+t+2+pk')\}\notag\end{align}

\begin{lemma}
Let $f(x)$ be an even function and $g(x)$ be an odd function, then
\begin{align} &\quad\sum_{k\ge 0}f(k+1)g(k+1)\{\delta(k-t-pk')-\delta(k+t+2+pk')\}\notag\\
&=f(pk'+t+1)g(pk'+t+1)\notag\end{align}
\end{lemma}

The proof of this lemma will be given later. Now we apply it to
calculate the innermost summation of $S$.---- let
\[f(x)=\exp(\epsilon(-(t+1)^2+x^2)C/2)\] and \[g(x)=\sinh((s+1)Cx),\] we
have
\begin{align}
S&=4\sum_{ \substack {k'=-N\\ k'+N\ even}}^N \sum_{t=0}^{r-1}\exp(\alpha {k'}^2C/2+\beta
k'(t+1)C)\sinh((t+1)(l+1)C)\notag\\
&\times\exp(\epsilon pk'(pk'+2t+2)C/2)\sinh((s+1)(pk'+t+1)C)\notag\\
&=4\sum_{ \substack {k'=-N\\ k'+N\ even}}^N \sum_{t=0}^{r-1}\exp(qp{k'}^2C/2+qk'(t+1)C)\notag\\
&\times \sinh((t+1)(l+1)C)\sinh((s+1)(pk'+t+1)C)\notag\\
&=4\sum_{t=0}^{r-1}\sinh((t+1)(l+1)C)\notag\\
&\times\sum_{ \substack {k'=-N\\ k'+N\ even}}^N \exp(qp{k'}^2C/2+qk'(t+1)C)\sinh((s+1)(pk'+t+1)C)\notag\end{align}

Note that $\alpha+\epsilon p^2=qp$ and $\beta+\epsilon p=q$.

\begin{lemma}
\begin{align} &\sum_{ \substack {k'=-N\\ k'+N\ even}}^N \exp(\alpha_1{k'}^2+\alpha_2k')\sinh(\alpha_3k'+\zeta)\notag\\
&=\sum_{ \substack {k'=-N\\ k'+N\ even}}^N \exp(\alpha_1{k'}^2+\alpha_3k')\sinh(\alpha_2k'+\zeta)\notag\end{align}
\end{lemma}

Apply this lemma to the inner summation of (3.4), we have

\begin{align}
S&=4\sum_{t=0}^{r-1}\sinh((t+1)(l+1)C)\notag\\
&\times\sum_{ \substack {k'=-N\\ k'+N\ even}}^N  \exp(qp{k^\prime}^2C/2+qk'(t+1)C)\sinh(pk'(s+1)C+(s+1)(t+1)C)\notag\\
&=4\sum_{t=0}^{r-1}\sinh((t+1)(l+1)C)\notag\\
&\times\sum_{ \substack {k'=-N\\ k'+N\ even}}^N  \exp(qp{k^\prime}^2C/2+pk'(s+1)C)\sinh(qk'(t+1)C+(s+1)(t+1)C)\notag\\
&=4\sum_{ \substack {k'=-N\\ k'+N\ even}}^N  \exp(qp{k^\prime}^2C/2+pk'(s+1)C)\notag\\
&\times\sum_{t=0}^{r-1}\sinh((t+1)(l+1)C)\sinh((t+1)(qk'+s+1)C)\notag\\
&=2\sum_{ \substack {k'=-N\\ k'+N\ even}}^N  \exp(qp{k^\prime}^2C/2+pk'(s+1)C)\notag\\
&\times\sum_{t=0}^{r-1}(\cosh((t+1)(l+s+2+qk')C)-\cosh((t+1)(l-s-qk')C))\notag\end{align}

our calculation will be done with the help of the following lemma
\begin{lemma} Suppose r+1 is a prime number, $\eta$ an integer, and $r+1\nmid\eta$, then
\[\sum_{t=0}^{r-1}\cosh((t+1)\eta C)=-\frac{1+cos(\eta\pi)}2\]
where $C=\frac{\pi i}{r+1}$
\end{lemma}

So if both $l+s+2+qk'$ and $l-s-qk'$ are nonzero, by choosing a sufficiently large prime number r+1, we have
\[\sum_{t=0}^{r-1}(\cosh((t+1)(l+s+2+qk')C)-\cosh((t+1)(l-s-qk')C)=0\]

if $l+s+2+qk'=0$, then $l-s-qk'$ is nonzero, since the sum of the two is $2l+2$, which is always positive and even, in this case
\[\sum_{t=0}^{r-1}(\cosh((t+1)(l+s+2+qk')C)-\cosh((t+1)(l-s-qk')C)=r-(-1)=r+1\]

if $l-s-qk'=0$, the situation is similar, and we get
\[\sum_{t=0}^{r-1}(\cosh((t+1)(l+s+2+qk')C)-\cosh((t+1)(l-s-qk')C)=-1-r\]

in summary,

\begin{align} &\sum_{t=0}^{r-1}(\cosh((t+1)(l+s+2+qk')C)-\cosh((t+1)(l-s-qk')C))\notag\\
&=-(r+1)\{\delta (l-s-qk')-\delta (l+s+2+qk')\}\notag\end{align}

so
\begin{align} S&=-2\sum_{ \substack {k'=-N\\ k'+N\ even}}^N \exp(qp{k^\prime}^2C/2+pk'(s+1)C)\notag\\
&\times(r+1)\{\delta (l-s-qk')-\delta (l+s+2+qk')\}\notag\end{align}

Finally from (A) and the above result we have

\begin{align} (-1)^{l+s}g_\sigma^l(q,p;s)&(-2(r+1))=(-1)^{qN}(-1)^{pN}\times S\notag\\
&=(-1)^{pN+qN}(-2(r+1))\sum_{ \substack {k'=-N\\ k'+N\ even}}^N \exp(qp{k^\prime}^2C/2+pk'(s+1)C)\notag\\
&\times\{\delta(l-s-qk^{'})-\delta(l+s+2+qk^{'})\}\notag\end{align}

cancel (-2(r+1)) from both sides, we got two polynomials which are equal at infinitely many points, i.e., the primitive
4(r+1)-th root of unity, hence the two polynomials are equal. Note further that $g_\sigma^l(q,p;s)=0$ unless $l+s+qN$ is
even--due to the $\delta$ function, thus one can check that the sign factor before the summation takes the form as we expected,
which is $(-1)^{pN}$. And our formula for $g^{,}$s is proved.

\end{proof}

\section{Appendix}\label{apendix}
\label{appendix}
Now we will present a proof of lemma 3 and lemma 4.

\begin{proof}[Proof of lemma 3]
By change of variables, we have
\[\sum_{k\ge 0}f(k+1)g(k+1)\delta(k-t-pk')=\sum_{k^{''}\ge 1}f(k^{''})g(k^{''})\delta(k^{''}-t-1-pk')\]
and
\begin{align}\sum_{k\ge 0}f(k+1)g(k+1)\delta(k+t+2+pk')&=\sum_{k^{''}\ge 1}f(k^{''})g(k^{''})\delta(k^{''}+t+1+pk')\notag\\
&=\sum_{k^{''}\le -1}f(-k^{''})g(-k^{''})\delta(-k^{''}+t+1+pk')\notag\\
&=-\sum_{k^{''}\le -1}f(k^{''})g(k^{''})\delta(k^{''}-t-1-pk')\notag\end{align}
So
\begin{align}&\sum_{k\ge 0}f(k+1)g(k+1)\{\delta(k-t-pk')-\delta(k+t+2+pk')\}\notag\\
                       &=\sum_{k^{''}\ne 0}f(k^{''})g(k^{''})\delta(k^{''}-t-1-pk')\notag\\
                       &=\sum_{k^{''}\in Z}f(k^{''})g(k^{''})\delta(k^{''}-t-1-pk')\notag\\
                       & since\ f(0)=0\notag\\
                       &=f(pk'+t+1)g(pk'+t+1)\notag\end{align}
\end{proof}

\begin{proof}[Proof of lemma 4]
\begin{align}
&\sum_{ \substack {k'=-N\\ k'+N\ even}}^N \exp(\alpha_1{k'}^2+\alpha_2k')\sinh(\alpha_3k'+\zeta)\notag\\
&=\sum_{ \substack {k'=-N\\ k'+N\ even}}^N \exp(\alpha_1{k'}^2+\alpha_2k')(\sinh(\alpha_3k')\cosh(\zeta)+\cosh(\alpha_3k')\sinh(\zeta))\notag
\end{align}

And
\begin{align}
&\sum_{ \substack {k'=-N\\ k'+N\ even}}^N \exp(\alpha_1{k'}^2+\alpha_2k')\sinh(\alpha_3k')\notag\\
&=\sum_{ \substack {k'=-N\\ k'+N\ even}}^N \exp(\alpha_1(-k')^2-\alpha_2k')\sinh(-\alpha_3k')\notag\\
&=\sum_{ \substack {k'=-N\\ k'+N\ even}}^N -\exp(\alpha_1{k'}^2-\alpha_2k')\sinh(\alpha_3k')\notag\\
&=1/2\sum_{ \substack {k'=-N\\ k'+N\ even}}^N \exp(\alpha_1{k'}^2)\sinh(\alpha_3k')(\exp(\alpha_2k')-\exp(-\alpha_2k'))\notag\\
&=\sum_{ \substack {k'=-N\\ k'+N\ even}}^N \exp(\alpha_1{k'}^2)\sinh(\alpha_3k')\sinh(\alpha_2k')\notag
\end{align}

Apply this trick to
\[\sum_{ \substack {k'=-N\\ k'+N\ even}}^N \exp(\alpha_1{k'}^2+\alpha_2k')\cosh(\alpha_3k')\]
we have
\[\sum_{ \substack {k'=-N\\ k'+N\ even}}^N \exp(\alpha_1{k'}^2+\alpha_2k')\cosh(\alpha_3k')=\sum_{ \substack {k'=-N\\ k'+N\ even}}^N \exp(\alpha_1{k'}^2)\cosh(\alpha_3k')\cosh(\alpha_2k')\]

So \begin{align}
&\sum_{ \substack {k'=-N\\ k'+N\ even}}^N \exp(\alpha_1{k'}^2+\alpha_2k')\sinh(\alpha_3k'+\zeta)\notag\\
&=\sum_{ \substack {k'=-N\\ k'+N\ even}}^N \exp(\alpha_1{k'}^2)(\sinh(\alpha_2k')\sinh(\alpha_3k')\cosh(\zeta)+\cosh(\alpha_2k')\cosh(\alpha_3k')\sinh(\zeta))\notag
\end{align}

In this expression, $\alpha_2$ and $\alpha_3$ are symmetric, and our formula follows.
\end{proof}

\end{document}